\documentclass[11pt,a4paper]{amsart}

\usepackage{amssymb}
\usepackage[final]{showkeys}
\usepackage{microtype}
\usepackage{color}
\usepackage{graphicx}
\usepackage[hmargin=3cm,vmargin={5cm,4cm}]{geometry}

\newtheorem{te}{Theorem}[section]

\newtheorem{prop}[te]{Proposition}
\newtheorem{lem}[te]{Lemma}
\newtheorem{coro}[te]{Corollary}

\numberwithin{equation}{section}

\allowdisplaybreaks

\begin{document}

\title[]{The existence and local uniqueness of the eigenfunctions of the non-linear operator $\Delta_H u^{n}$ in the hyperbolic Poincaré half-plane.}	
	\author{F. Maltese}
	\maketitle
	
\begin{abstract}
In this article, we find locally eigenfunctions for a particular nonlinear hyperbolic differential operator $\Delta_H u^{n}$, where $\Delta_H$ is the hyperbolic Laplacian in the Poincairé half-plane.
We have proved that these eigenfunctions are an analytic and non-exact, whose coefficients satisfy a specific algebraic recursive rule.
The existence of these eigenfunctions allows us to find non-exact solutions with respect to the spatial coordinate of nonlinear diffusive PDEs on the Poincaré half-plane, which could describe a possible one-dimensional physical model.

	\bigskip

\end{abstract}
\maketitle	

\section*{Contents}

\subsection*{0.Introduction .........................................................................................................................1}

\subsection*{1.Transformation of the operator $ \Delta_H u^{n}$ into an algebraic differential operator by change of variables. .................................................................................................................3}

\subsection*{2.Solving (1.6) using formal power series .................................................................................4}

\subsection*{3.The demonstration that there doesn't exist a polynomial function with respect to the variable z as a solution of (1.6) .......................................................................................8}

\subsection*{3.1. Local existence and uniqueness of an analytic solution of (1.6) .........................................10}

\subsection*{3.2 The solution of the equation
	(0.3) ......................................................................................12}

\subsection*{4.Eigenfunction of $ \Delta_H u^{n}$ like a solution of nonlinear hyperbolic partial differential equations ..........................................................................................................................12}

\subsection*{References
	..............................................................................................................................15}

\section*{Introduction}

In this paper, it will be proven the existence and local uniqueness of an analytic eigenfunction of the nonlinear hyperbolic Laplacian operator

\begin{equation}
	\Delta_H u^{n}=\frac{1}{\sinh \eta}\frac{\partial}{\partial \eta}\sinh \eta \frac{\partial  u^{n}}{\partial \eta}  \quad n>1,
\end{equation}

in the hyperbolic half-plane by hyperbolic polar coordinates. In the present case, the focus will be on the $\eta$-coordinate, which indicates the length of the geodesic in the hyperbolic half-plane of Poincaré. For more information this, see [4, Section 2] .

\quad 

The search for these eigenfunctions was motivated by the finding of solutions to the nonlinear diffusive equation on the Poincaré half-plane

\begin{equation}
\widehat{O}_t  u (\eta,t)= \Delta_H  u^{n}= \frac{1}{\sinh \eta}\frac{\partial}{\partial \eta}\sinh \eta \frac{\partial  u^{n}}{\partial \eta}-u, \quad n>1, \hspace{0.2cm} \eta>0.
\end{equation}	

With $	\widehat{O}_t$ a linear time operator. This equation is a variation of an HPME (Hyperbolic porous media equation) $$\widehat{O}_t  u (\eta,t)= \Delta_H  u^{n} $$ with the addition of a reactive term '$-u$' (for HPME, see [11]). Using the methods of separation of variables and invariant spaces (for these methods it is possible to see [3] and [6, Section 6]), in the article [2], it was found exact solutions of the type $u(\eta,t)=f(t)\sqrt[n]{c_1\ln(\tanh \frac{\eta}{2})+c_2}$. Where $u(\eta,t)$ must satisfy the following equations $$\widehat{O}_t f(t)=-f(t),\hspace{0.1cm} \Delta_H\Bigg({\sqrt[n]{c_1\ln(\tanh \frac{\eta}{2})+c_2}\Bigg)}^{n}=0$$ i.e., $\sqrt[n]{c_1\ln(\tanh \frac{\eta}{2})+c_2} \in Ker \Delta_H g^n(\eta)$.
By the same method, we found locally an analytic solution of the type $u(\eta,t)=f(t)g(\eta)$ such that $\Delta_H g^n(\eta)=\lambda g(\eta)$, for some $\lambda \neq 0 \in \mathbb R$ and $\widehat{O}_t f(t)=\lambda f^n(t) -f(t)$ such that $g(\eta)$ is an eigenfunction of $\Delta_H u^n$. Note that, unlike the exact solutions that cancel out the operator $\Delta_H u^{n}$, the eigenfunction $g(\eta)$ is not only the spatial part of the solution to (0.1) but also the spatial part of the HPME solution described above, and what is interesting is that we have found a unique local analytical solution. Indeed, when looking for solutions to an HPME, one usually studies weak solutions, obtains numerical approximations of the solution, or performs a qualitative analysis of the solution. However, one can always refer to [11] for more details on these facts.

In particular, these last solutions will be explicitly calculated also in the temporal part $f(t)$, in section 4, as solutions of variants of (0.2) i.e., the PDE nonlinear (4.4) and (4.6) (for information on the last equation, see [2]).
\quad 
\quad

Therefore, in most of this article, we are going to be focusing on the study of the solutions of the ODE

\begin{equation}
	 \frac{1}{\sinh \eta}\frac{\partial}{\partial \eta}\sinh \eta \frac{\partial  u^{n}(\eta)}{\partial \eta}=\lambda u(\eta).
\end{equation}

And this will be done by several intermediate steps, which will be illustrated in the different sections.
In the first section, by transforming the $\eta$-coordinate (1.3), we have transformed the operator (0.1) into an algebraic operator, i.e., the coefficients of the differential operator are polynomials.

\quad

In the second section we have found a formal solution of the PDE (1.6) modified by the transformation (1.3). That is, we found a formal power series that satisfies (1.6) by recursive algebraic formulas of the coefficients (2.4) and (2.5). Afterwards, we have proved that formula (2.5) depends on the choice of the first two coefficients of the sequence associated with the formal power series .

\quad

In the first part of the third section, we have shown that there is no polynomial solution of (1.6) studying thr leading coefficient of the polynomial, transformed by algebraic operator (1.5) and applying the equation (1.6). In the second part of the third section, we have found locally an analytic solution of (1.6), and, then in the last part of this section, returning to equation (0.3) via (1.3), we have found an eigenfunction of the operator (0.1).
In particular, in this section, the analytic function of (1.6) enjoys a recursive property among its derivatives (3.14), and this fact is interesting because we can obtain approximations of the solution, using an algorithm that implements the recursive formula (3.14) obtained. This recursive nature of the the solution's coefficients is due to the fact that the differential operator, $ \Delta_H u^{n}$ becomes algebraic under transformations, a non-trivial fact.

\quad

\section{Transformation of the operator  $ \Delta_H u^{n}$ into an algebraic differential operator by change of variables.}

A first step in solving (0.3) is to implement a variable change to make the operator (0.1) polynomial. For that, we consider as test functions the compound functions of the type $f(z(\eta))$, where $z=z(\eta)$ is the change of variable. By the chain rule we have that $$\frac{\partial f(z(\eta))}{\partial \eta}=\frac{\partial f(z)}{\partial z}\frac{\partial z(\eta)}{\partial \eta}.$$

\quad

So we have that $$\frac{1}{\frac{\partial z(\eta)}{\partial \eta}}\frac{\partial f(z(\eta))}{\partial \eta}=\frac{\partial f(z)}{\partial z}.$$ If we assume that 

\begin{equation}
\frac{\partial z(\eta)}{\partial \eta}={\delta}^{'}\sinh \eta,
\end{equation}

where ${\delta}^{'}$is a suitable positive real parameter to be determined in order to find the solution to (0.3), we obtain that

$$
	\frac{1}{\sinh \eta}\frac{\partial f(z(\eta))}{\partial \eta}={\delta}^{'}\frac{\partial f(z)}{\partial z},
$$

that is, the following equality between differential operators

\begin{equation}
		\frac{1}{\sinh \eta}\frac{\partial}{\partial \eta}={\delta}^{'}\frac{\partial}{\partial z}.
\end{equation}

Then the transformation $z(\eta)$ must satisfy equation (1.1). In particular, we choose

\begin{equation}
	z(\eta)={\delta}^{'}\cosh \eta.
\end{equation}

\quad

We complete the transformation of the operator (0.1) by modifying the following differential operator $\sinh \eta \frac{\partial  (\bullet)^{n}}{\partial \eta}$, again using the test function $f(z(\eta))$.

$$\sinh \eta \frac{\partial f^{n}(z(\eta)) }{\partial \eta}=\frac{{\sinh \eta}^{2}}{\sinh \eta}\frac{\partial f^{n}(z(\eta)) }{\partial \eta}$$

$$\sinh \eta \frac{\partial f^{n}(z(\eta)) }{\partial \eta}=({\cosh \eta}^{2}-1)	\frac{1}{\sinh \eta}\frac{\partial f^{n}(z(\eta))}{\partial \eta}$$

$$\sinh \eta \frac{\partial f^{n}(z(\eta)) }{\partial \eta}=\frac{1}{{\delta}^{'}}\Big(z^{2}-{{\delta}^{'}}^{2}\Big)\frac{\partial f^{n}(z) }{\partial z}.$$

This gives us the following equality between operators

\begin{equation}
	\sinh \eta \frac{\partial  (\bullet)^{n}}{\partial \eta}=\frac{1}{{\delta}^{'}}\Big(z^{2}-{{\delta}^{'}}^{2}\Big)\frac{\partial(\bullet)^{n} }{\partial z}.
\end{equation}

\quad

If we combine (1.2) with (1.4), we obtain the complete transformation of the operator (0.1) as follows

\begin{equation}
	\frac{1}{\sinh \eta}\frac{\partial}{\partial \eta}\sinh \eta \frac{\partial  (\bullet)^{n}}{\partial \eta}=\frac{\partial }{\partial z}\Big(z^{2}-{{\delta}^{'}}^{2}\Big)\frac{\partial(\bullet)^{n} }{\partial z}.
\end{equation}

\quad

If we want to find as solutions of (0.3) a function like the test function $f(z(\eta))$, i.e., of the type $u(z(\eta))$. Then, via (1.5), equation (0.3) becomes

\begin{equation}
	\frac{\partial }{\partial z}\Big(z^{2}-{{\delta}^{'}}^{2}\Big)\frac{\partial u^{n}(z) }{\partial z}=\lambda u(z).
\end{equation}

\quad

\section{Solving (1.6) using formal power series }

Given a formal power series $S(z)=\sum_{i \geq 0}a_{i}z^{i}$,  associated with an infinite sequence $ a_{i}  \in \mathbb{R}$, we want to check which constraints have the terms of the sequence $a_{i}$ when we substitute $S(z)$ into (1.6). First, we substitute $S(z)$ into (1.5)

$$ \hspace{-6cm}	\frac{\partial }{\partial z}\Big(z^{2}-{{\delta}^{'}}^{2}\Big)\frac{\partial S^{n}(z) }{\partial z} =\frac{\partial }{\partial z}\Big(z^{2}-{{\delta}^{'}}^{2}\Big)\frac{\partial \Big(\sum_{i \geq 0}a_{i}z^{i}  \Big)^{n}  }{\partial z}$$

$$\hspace{-2.5cm}=\frac{\partial }{\partial z}\Big(z^{2}-{{\delta}^{'}}^{2}\Big)\frac{\partial\Big( \sum_{i \geq 0}b_{i}z^{i}\Big)  }{\partial z},$$

where 

$$\Bigg(\sum_{i \geq 0}a_{i}z^{i}  \Bigg)^{n}=\sum_{i \geq 0}b_{i}z^{i}  $$

such that the sequence $b_{i}$ satisfies the following recursive formula
\begin{equation}
b_{i}=\frac{1}{ia_{0}}\sum_{k= 1}^{i}(k(n+1)-i)a_{k}b_{i-k}\hspace{0.2cm}with\hspace{0.1cm}a_{0}\neq0\hspace{0.1cm} and \hspace{0.1cm} i>0.
\end{equation}

While for $i=0$ we'll have $b_{0}=a_{0}^{n}$.See [12], for the details on how to obtain these recursive formulas.

\quad

Now, we can continue to perform calculations on the formal series raised to a power

$$	\frac{\partial }{\partial z}\Big(z^{2}-{{\delta}^{'}}^{2}\Big)\frac{\partial S^{n}(z) }{\partial z}=\frac{\partial }{\partial z}\Big(z^{2}-{{\delta}^{'}}^{2}\Big)\Bigg(\sum_{i \geq 1}ib_{i}z^{i-1}\Bigg)$$

$$\hspace{4.5cm}=\frac{\partial }{\partial z}\Bigg( \sum_{i \geq 1}ib_{i}z^{i+1}-{{\delta}^{'}}^{2}\sum_{i \geq 1}ib_{i}z^{i-1}\Bigg)$$

$$\hspace{5.2cm}=\sum_{i \geq 1}i(i+1)b_{i}z^{i}-{{\delta}^{'}}^{2}\sum_{i \geq 2}i(i-1)b_{i}z^{i-2}$$

$$\hspace{6.2cm}=\sum_{i \geq 1}i(i+1)b_{i}z^{i}-{{\delta}^{'}}^{2}\sum_{i \geq 0}(i+2)(i+1)b_{i+2}z^{i}$$

\begin{equation}
\hspace{4.5cm}=\sum_{i \geq 0}(i+1)(ib_{i}-{{\delta}^{'}}^{2}(i+2)b_{i+2})z^{i}.
\end{equation}
	
In order for the formal power series $S(z)$ to satisfy equation (1.6), i.e. to be an eigenfunction, it must be verified that

$$
\sum_{i \geq 0}(i+1)(ib_{i}-{{\delta}^{'}}^{2}(i+2)b_{i+2})z^{i}=\lambda\sum_{i \geq 0}a_{i}z^{i}. 
$$

By equating every monomial in both series, we arrive at the following formula:

\begin{equation}
\hspace{2.5cm} (i+1) (ib_{i}-{{\delta}^{'}}^{2}(i+2)b_{i+2})=\lambda a_{i}.        
\end{equation}

Using the recursive formula (2.1), we can partially expand the first member of equation (2.3) in terms of the coefficients $a_{i}$ for $i>0$ .

$$
(i+1)\Big(i\frac{1}{ia_{0}}\sum_{k= 1}^{i}(k(n+1)-i)a_{k}b_{i-k}-{{\delta}^{'}}^{2}(i+2)\frac{1}{(i+2)a_{0}}\sum_{k= 1}^{i+2}(k(n+1)-i-2)a_{k}b_{i+2-k}\Big)=\lambda a_{i}.        
$$

$$\hspace{1.5cm}(i+1)\Big(\frac{1}{a_{0}}\sum_{k= 1}^{i}(k(n+1)-i)a_{k}b_{i-k}-{{\delta}^{'}}^{2}\frac{1}{a_{0}}\sum_{k= 1}^{i+2}(k(n+1)-i-2)a_{k}b_{i+2-k}\Big)=\lambda a_{i}. 
$$

\begin{equation}\frac{i+1}{a_{0}}\Big(\sum_{k= 1}^{i}(k(n+1)-i)a_{k}b_{i-k}-{{\delta}^{'}}^{2}\sum_{k= 1}^{i+2}(k(n+1)-i-2)a_{k}b_{i+2-k}\Big)=\lambda a_{i}. 
\end{equation}

Now, we make explicit the equation in respect with the term $a_{i+2}$ to rewrite (2.4)
in recursive way.

$$-{{\delta}^{'}}^{2}(i+1)(i+2)na_{i+2}a_{0}^{n-1}+ \frac{i+1}{a_{0}}\Big(\sum_{k= 1}^{i}(k(n+1)-i)a_{k}b_{i-k}-{{\delta}^{'}}^{2}\sum_{k= 1}^{i+1}(k(n+1)-i-2)a_{k}b_{i+2-k}\Big)=\lambda a_{i}.$$

\begin{equation}	
a_{i+2}=-\frac{\lambda a_{i}}{{{\delta}^{'}}^{2}(i+1)(i+2)na_{0}^{n-1}} + \frac{1}{{{\delta}^{'}}^{2}(i+2)na_{0}^{n}}\Big(\sum_{k= 1}^{i}(k(n+1)-i)a_{k}b_{i-k}-{{\delta}^{'}}^{2}\sum_{k= 1}^{i+1}(k(n+1)-i-2)a_{k}b_{i+2-k}\Big).
\end{equation}

\quad

\quad

\quad

	Examining formula (2.5), one can argue that it can be applied recursively to each term of the second member of (2.5)$a_k$, for $k\ge 2$. Thus, (2.5) depends only on the first two terms of the sequence  $a_k$, $a_0$ and $a_1$. Let us try to see how. First, however, we must prove the following Lemma.

\begin{lem}
		The sequences $b_{j}$ of (2.1)  can be viewed as rational algebraic functions of the $a_{i}$ terms, where $0\leq i \leq j$ i.e., 
	
	$$b_{j}=b_{j}(a_{0},...,a_{l},...,a_{j}), \hspace{0.1cm} b_{j} \in \mathbb{R}(a_{0},...,a_{l},...,a_{j}) $$ 
	 
	which, to make notation more compact, this can be denoted as follows:
	
	\begin{equation}
b_{j}\Big({(a_{l})}^{j}_{l=0}\Big):=	b_{j}(a_{0},...,a_{l},...,a_{j}).
	\end{equation}
\end{lem}
\begin{proof}
	We prove the lemma using the induction for j.
	
	\quad
	
	Let's begin with the base of the induction i.e., for j=0.
	
	 Since we know that $b_{0}=a_{0}^{n}$, the base case is trivial.
	 
	 Now we consider the inductive hypothesis that is $$	b_{k}=b_{k}\Big({(a_{l})}^{k}_{l=0}\Big) \hspace{0.1cm} for \hspace{0.1cm} 0\leq k \leq j-1.$$
	 
	 Considering the recursive formula (2.1), we obtain
	 
	 $$b_{j}=\frac{1}{ja_{0}}\sum_{k= 1}^{j}(k(n+1)-j)a_{k}b_{j-k}\Big({(a_{l})}^{j-k}_{l=0}\Big)$$
	 
	  $$b_{j}=b_{j}\Big({(a_{l})}^{j}_{l=0}\Big).$$
	
	Thus, we have proven the thesis.
	
\end{proof}

Now we can prove the following proposition.

\begin{prop}
	
Assuming that we know the first two terms of the sequence $a_{0}$ and $a_{1}$, then the formula (2.5) can be rewritten as follows

\begin{equation}
	a_{i+2}=\mathcal{Q}_{i+2,n}(a_{0},a_{1}, \lambda,{\delta}^{'}) \hspace{0.2cm}\forall \hspace{0.1cm} k\geq 0.
\end{equation}
Where $\mathcal{Q}_{0,n}(a_{0},a_{1}, \lambda,{\delta}^{'}):=a_{0}$ , $\mathcal{Q}_{1,n}(a_{0},a_{1}, \lambda,{\delta}^{'}):=a_{1}$ and $\mathcal{Q}_{j,n}(X,Y,Z,W) \in \mathbb{R}(X,Y,Z,W) \hspace{0.4cm} \forall j \geq 0$ i.e., are rational algebraic functions in $\mathbb{R}^{4}$. In particular, we have that $\mathcal{Q}_{0,n}(X,Y,Z,W)=X$ and $\mathcal{Q}_{1,n}(X,Y,Z,W)=Y$.

\end{prop}
\begin{proof}
 We are going to prove the proposition by induction on the index $i$.

\quad

\title{INDUCTION BASE: i=0} 

\quad

In this case, we use the formula (2.3) for $i=0$ i.e.

$$ -2{{\delta}^{'}}^{2}b_{2}=\lambda a_{0}. $$

Now, we will calculate $a_{2}$ using the above equation. The procedure is the same as that used to derive the general formula (2.3) for $i>0$.

In fact, using the formula (2.1) for $b_{2}$, we obtain

$$-2{{\delta}^{'}}^{2}na_{2}a_{0}^{n-1} -\frac{{{\delta}^{'}}^{2}}{a_{0}}(n-1)a_{1}b_{1}=\lambda a_{0}.$$ 

Applying formula (2.1) again to $b_{1}$, we obtain

$$-2{{\delta}^{'}}^{2}na_{2}a_{0}^{n-1} -{{\delta}^{'}}^{2}n(n-1)a_{1}^{2}a_{0}^{n-2}=\lambda a_{0},$$

since $b_{1}=na_{1}a_{0}^{n-1}$.

Now, making explicit respect $a_{2}$, we find that

\begin{equation}
a_{2}= -\frac{\lambda}{2{{\delta}^{'}}^{2}na_{0}^{n-2}} -\frac{(n-1)a_{1}^{2}}{2a_{0}}.
\end{equation}

Thus,

$$ a_{2}=\mathcal{Q}_{2,n}(a_{0},a_{1}, \lambda,{\delta}^{'}),$$

with

\begin{equation}
	\mathcal{Q}_{2,n}(X,Y, Z,W)=-\frac{Z}{2W^{2}nX^{n-2}}-\frac{(n-1)Y^{2}}{2X}.
\end{equation}

\quad

\title{INDUCTIVE HYPOTHESIS:   $i> 0$}.

\quad

We can now apply the inductive hypothesis to (2.5) by setting $\mathcal{Q}_{i,n}(a_{0},a_{1}, \lambda,{\delta}^{'})=\mathcal{Q}_{i,n}$ to simplify the notation. Using the Lemma (2.1), we get

$$
	a_{i+2}=-\frac{\lambda \mathcal{Q}_{i,n}}{{{\delta}^{'}}^{2}(i+1)(i+2)na_{0}^{n-1}} + \frac{1}{{{\delta}^{'}}^{2}(i+2)na_{0}^{n}}\Big(\sum_{k= 1}^{i}(k(n+1)-i)\mathcal{Q}_{k,n}b_{i-k}\Big((\mathcal{Q}_{l,n})_{l=0}^{i-k}\Big)\Big)$$
	
	$$\hspace{3.5cm}-\frac{1}{{{\delta}^{'}}^{2}(i+2)na_{0}^{n}}\Big({{\delta}^{'}}^{2}\sum_{k= 1}^{i+1}(k(n+1)-i-2){Q}_{k,n}b_{i+2-k}\Big((\mathcal{Q}_{l,n})_{l=0}^{i+2-k}\Big)\Big).
$$

From this last equation, we can deduce the thesis

$$ 	a_{k+2}=\mathcal{Q}_{k+2,n}(a_{0},a_{1}, \lambda,{\delta}^{'}).$$

\end{proof}

\quad

As a consequence of Proposition (2.1) we obtain a recursive formula to calculate the polynomials $\mathcal{Q}_{i,n}(X,Y,Z,W)=\mathcal{Q}_{i,n}.$

\begin{equation}
\mathcal{Q}_{i+2,n}=-\frac{Z \mathcal{Q}_{i,n}}{{W}^{2}(i+1)(i+2)nX^{n-1}} + \frac{1}{W^{2}(i+2)nX^{n}}\Big(\sum_{k= 1}^{i}(k(n+1)-i)\mathcal{Q}_{k,n}b_{i-k}\Big((\mathcal{Q}_{l,n})_{l=0}^{i-k}\Big)\Big)
\end{equation}

$$ \hspace{3.5cm}-\frac{1}{W^{2}(i+2)nX^{n}}\Big(W^{2}\sum_{k= 1}^{i+1}(k(n+1)-i-2){Q}_{k,n}b_{i+2-k}\Big((\mathcal{Q}_{l,n})_{l=0}^{i+2-k}\Big)\Big)	       . $$

\quad

\section{The demonstration that doesn't exist a polynomial function with respect to the variable z as solution of (1.6) }

In this section, we will prove that there is no exact polynomial function $P(z)$ as a solution to (1.6).

We remark that if the degree of $P(z)$ is equal to $m>0$ i.e. $deg P(z)=m$ then we obtain that

 $$ deg \Bigg(\frac{\partial }{\partial z}\Big(z^{2}-{{\delta}^{'}}^{2}\Big)\frac{\partial P^{n}(z) }{\partial z}\Bigg)=deg\Bigg(\Big(z^{2}-{{\delta}^{'}}^{2}\Big)\frac{\partial P^{n}(z) }{\partial z}\Bigg)-1=2+ deg\Bigg( \frac{\partial P^{n}(z) }{\partial z} \Bigg)-1 =1+deg\Bigg(\frac{\partial P^{n}(z) }{\partial z} \Bigg)$$ 
 $$\hspace{13.75cm}=1+mn-1=mn .$$

So 

\begin{equation}
	deg \Bigg(\frac{\partial }{\partial z}\Big(z^{2}-{{\delta}^{'}}^{2}\Big)\frac{\partial P^{n}(z) }{\partial z}\Bigg)=mn \hspace{0.2cm} if  \hspace{0.2cm} deg P(z)=m .
\end{equation}
 
Since we want to find a non-trivial solution of (1.6), then $m \geq 1$ and we know that $n>1$, so $m < mn$.

Now let's suppose by contradiction that $P(z)$ is a solution to (1.6), i.e.,

$$\frac{\partial }{\partial z}\Big(z^{2}-{{\delta}^{'}}^{2}\Big)\frac{\partial P^{n}(z) }{\partial z}=\lambda P(z).$$

Then the leading coefficient of $P(z)=\sum_{i=0}^{m}a_{i}z^{i}$, which we denote $l_{m}(P):=a_{m}$, is not zero.

\quad

Therefore, the coefficients $a_{i}$ of $P(z)$ satisfy the equation (1.6). In particular, let's compute the leading coefficient of $$R(z):=\frac{\partial }{\partial z}\Big(z^{2}-{{\delta}^{'}}^{2}\Big)\frac{\partial P^{n}(z) }{\partial z}.$$

\quad
First at all, we recall the fact that, given two polynomials $Q(z)$ and $Q'(z)$ with $deg(Q(z))=n', deg(Q'(z))=m'$ then $l_{n'+m'}(QQ')=l_{n'}(Q)l_{m'}(Q')$. This can be argued from the well-known formula for the coefficients of the product of two polynomials.

\begin{equation}
c_{i}=\sum_{k=0}^{i}a_{k}b_{i-k},
\end{equation} 
if $Q(z)=\sum_{i=0}^{n'}a_{i}z^{i}, Q'(z)=\sum_{i=0}^{m'}b_{i}z^{i}$ and $Q(z)Q'(z)=\sum_{i=0}^{n'+m'}c_{i}z^{i}$.

\quad

Using the inductive method from (3.2), we can prove that

\begin{equation}
l_{nm}(P^{n})=(l_{m}(P))^{n}.
\end{equation}

Afterwards, we get that

$$ l_{nm-1}\Big(\frac{\partial P^{n}}{\partial z}\Big)=nm(l_{m}(P))^{n}.$$

\quad 

Now we find leading coefficient with product polynomial $\Big(z^{2}-{{\delta}^{'}}^{2}\Big)\frac{\partial P^{n}(z) }{\partial z}$,

$$ l_{nm+1}\Big(\Big(z^{2}-{{\delta}^{'}}^{2}\Big)\frac{\partial P^{n} }{\partial z}\Big)=l_{2}\Big(z^{2}-{{\delta}^{'}}^{2}\Big)nm(l_{m}(P))^{n}=nm(l_{m}(P))^{n}.$$

\quad

Finally, we can calculate the leading coefficient of $R(z)$.
For the (3.1) we have that

$$l_{nm}(R)=l_{nm}\Big(\frac{\partial }{\partial z}\Big(z^{2}-{{\delta}^{'}}^{2}\Big)\frac{\partial P^{n} }{\partial z}\Big)=nm(nm+1)l_{m}(P)^{n}.$$

The solution $P(z)$ must satisfy (1.6). Since $mn>n$, this implies that

\begin{equation}
l_{nm}(R)=0 \rightarrow nm(nm+1)l_{m}(P)^{n}=0 \rightarrow l_{m}(P)=0.
\end{equation}

From (3.4) we derive a contradiction because, by hypothesis, we have supposed that $l_{m}(P)\neq 0$. Thus, we have proven that there is no exact polynomial solution.

\quad

\subsection{Local existence and uniqueness of an analytic solution of (1.6) }

\quad

In the previous section we saw that there are no polynomial solutions of (1.6), so we'll prove the local existence and uniqueness of an "infinite polynomial", i.e. an analytic solution represented as a power series. In order to do this, we rewrite (1.6) as a Cauchy problem of a first-order system ODE. First, we have to develop (1.6).

\begin{equation}
 2nzu^{n-1}\frac{\partial u}{\partial z}+n(n-1)(z^{2}-{{\delta}^{'}}^{2} )u^{n-2}\Bigg(\frac{\partial u}{\partial z}\Bigg)^{2}+n(z^{2}-{{\delta}^{'}}^{2} )u^{n-1}\frac{\partial^{2} u}{\partial z^{2}}=\lambda u.
\end{equation}

Now, we can make explicit the express (3.5) with respect to  $\frac{\partial^{2} u}{\partial z^{2}}$ and so it becomes

\begin{equation}
\frac{\partial^{2} u}{\partial z^{2}}=\frac{\lambda}{n(z^{2}-{{\delta}^{'}}^{2} )u^{n-2}} -\Bigg(\frac{\partial u}{\partial z}\Bigg)^{2}\frac{n-1}{u}-2\frac{\partial u}{\partial z}\frac{z}{z^{2}-{{\delta}^{'}}^{2}}.
\end{equation}
\quad
From (3.6), we can see that we need to find a solution in a suitable neighbourhood where the second member of (3.6) is well-defined, i.e. where $u\neq 0$ and $z\neq \pm {\delta}^{'}$.

If we set $\frac{\partial u}{\partial z}=u_{1}$, $\frac{\partial \tau}{\partial z}=1$, then (3.6) is transformed into

\begin{equation}	
\frac{\partial u_{1}}{\partial z}=\frac{\lambda}{n({(\tau-c)}^{2}-{{\delta}^{'}}^{2} )u^{n-2}} -{u_{1}}^{2}\frac{n-1}{u}-2u_{1}\frac{\tau-c}{{(\tau-c)}^{2}-{{\delta}^{'}}^{2}},
\end{equation}
where $c$ is a real constant, which will be determined by the choice of the initial conditions.
So we can choose an initial condition for them to simplify the (3.7), that is, if a solution is of the type $\bar{u}(z)=(\tau(z),u(z),u_{1}(z))$, $\bar{u}(0)=\Big(0, u^{(0)},u_{1}^{(0)}\Big)$, where $u^{(0)}\neq 0$ for as we have obtained (3.7).
Then we finally get the following Cauchy problem

\begin{equation}
	\begin{cases}
		\frac{\partial \bar{u} }{\partial z}=F(\tau, u,u_{1} )\\
		\\
		 \bar{u}(0)=\Big(0, u^{(0)},u_{1}^{(0)}\Big)
	\end{cases}	\\		
\end{equation}
$$whit \hspace{0.1cm} F(\tau, u,u_{1})=\Bigg(1,u_{1}, \frac{\lambda}{n({\tau}^{2}-{{\delta}^{'}}^{2} )u^{n-2}} -{u_{1}}^{2}\frac{n-1}{u}-2u_{1}\frac{\tau}{{\tau}^{2}-{{\delta}^{'}}^{2}} \Bigg). $$

\quad

In particular $F: D_{F}\longrightarrow \mathbb{R}^{3}$ where $D_{F}=\mathbb{R}^{3}-\{ (\pm {\delta}^{'}, 0, u_{1}) \mid \forall u_{1} \in \mathbb{R} \}$ is the maximal of $F$. It's immediately that in $D_{F}$ the function $F$ is analytic and that there exists, for a suitable $r>0$, an open ball with radius $r$ and centre $\bar{u}(0) \hspace{0.2cm} B_{r}(\bar{u}(0))\subset D_{F}$ where the Jacobian of $F$, $J_{F}(\tau,u,u_{1})$, is bounded. Thus, $F\mid_{B_{r}(\bar{u}(0))}$ is Lipschitzian with respect to the Euclidean metric on $\mathbb{R}^{3}$(See [9], page 84-85).

\quad

These conditions allow us to apply the theorem of existence and uniqueness the ODE to (3.6)(See [8]  pag 48-50). Thus, there exists an interval $I=[-{\delta}^{(1)},{\delta}^{(1)}]$ for a suitable ${\delta}^{(1)}>0$ such that in $I\times B_{r}\Big(0, u^{(0)},u_{1}^{(0)}\Big)$ there exists, and it's unique a solution of the type $$\bar{u}: (-{\delta}^{(2)},{\delta}^{(2)} ) \longrightarrow B_{r}\Big(0, u^{(0)},u_{1}^{(0)}\Big)$$
$$\hspace{2.8cm}z\longrightarrow \bar{u}(z)=(z,u(z),u_{1}(z)),$$

where $0<{\delta}^{(2)}\leq {\delta}^{(1)}$ so that $\bar{u} \in C^{1}((-{\delta}^{(2)},{\delta}^{(2)}))$. Furthermore, if we call $I_{{\delta}^{(2)}}=(-{\delta}^{(2)},{\delta}^{(2)} )$, unless we restrict $I_{\delta}^{(2)}$, we can assume that $\bar{u}(I_{{\delta}^{(2)}}) \subset B_{r}\Big(0, u^{(0)},u_{1}^{(0)}\Big)$, and this is allowed since $\bar{u}$ is continuous in the initial condition.

\quad

Thus, we are in the conditions to apply the 'Theorem of regularity of ODE' to the Cauchy problem (3.8)( See [5], page 84-85), i.e., unless we restrict $I_{{\delta}^{(2)}}$, $\bar{u}$ is an analytic function and in particular in $I_{{\delta}^{(2)}}$, we have that

\begin{equation}
	 u(z)=\sum_{i=0}^{+\infty}\frac{1}{i!}\frac{\partial^{i} u(0)}{\partial z^{i}}z^{i}.
\end{equation}

We recall that $u(z)$ is a solution of (1.6) and so the terms $\frac{1}{i!}\frac{\partial^{i} u(0)}{\partial z^{i}}$ satisfy the relation (2. 5), i.e., by setting $\partial_{i}u(0):=\frac{1}{i!}\frac{\partial^{i} u(0)}{\partial z^{i}}$ to simplify the following equation,

\begin{equation}
\partial_{i+2}u(0)=-\frac{\lambda \partial_{i}u(0)}{{{\delta}^{'}}^{2}(i+1)(i+2)nu^{n-1}(0)} + \frac{1}{{{\delta}^{'}}^{2}(i+2)nu^{n}(0)}\Bigg(\sum_{k= 1}^{i}(k(n+1)-i)\partial_{k}u(0)b_{i-k}\Bigg(\Big(\partial_{l}u(0)\Big)_{l=0}^{i-k}\Bigg)\Bigg)
\end{equation}

$$\hspace{4.5cm}-\frac{1}{{{\delta}^{'}}^{2}(i+2)nu^{n}(0)}\Bigg({{\delta}^{'}}^{2}\sum_{k= 1}^{i+1}(k(n+1)-i-2)\partial_{k}u(0)b_{i+2-k}\Bigg(\Big(\partial_{l}u(0)\Big)_{l=0}^{i+2-k}\Bigg)\Bigg).$$

\quad
\quad

However, since $u(z)$ is analytic in $I_{{\delta}^{(2)}}$, we have for every $a \in I_{{\delta}^{(2)}}$

\begin{equation}
	u(z)=\sum_{i=0}^{+\infty}\frac{1}{i!}\frac{\partial^{i} u(a)}{\partial z^{i}}(z-a)^{i}.
\end{equation}
 
 Therefore, the coefficients of (3.11) $\frac{1}{i!}\frac{\partial^{i} u(a)}{\partial z^{i}}$ satisfy a similar relation to (3.10) $\forall a \in I_{{\delta}^{(2)}}$ and by setting $\partial_{i}u:=\frac{1}{i!}\frac{\partial^{i} u(a)}{\partial z^{i}}$ we get
 
 \begin{equation}
 	\partial_{i+2}u=-\frac{\lambda \partial_{i}u}{{{\delta}^{'}}^{2}(i+1)(i+2)nu^{n-1}} + \frac{1}{{{\delta}^{'}}^{2}(i+2)nu^{n}}\Bigg(\sum_{k= 1}^{i}(k(n+1)-i)\partial_{k}ub_{i-k}\Bigg(\Big(\partial_{l}u\Big)_{l=0}^{i-k}\Bigg)\Bigg)
 \end{equation}
 $$\hspace{4.5cm}-\frac{1}{{{\delta}^{'}}^{2}(i+2)nu^{n}}\Bigg({{\delta}^{'}}^{2}\sum_{k= 1}^{i+1}(k(n+1)-i-2)\partial_{k}u b_{i+2-k}\Bigg(\Big(\partial_{l}u\Big)_{l=0}^{i+2-k}\Bigg)\Bigg).	         $$

 \subsection{The solution of the equation
 	 (0.3)}

 \quad
 \quad

Now we are ready to find the eigenfunction solution of the nonlinear operator (0.1). We have seen that the eigenfunction is of the type $u(z(\eta))$ with $z(\eta) \in ({{\delta}^{'}}, +\infty)$, since $\eta \in (0, +\infty)$ and $z(\eta)={{\delta}^{'}}\cosh \eta$ in $(0, +\infty)$ is strictly increasing. However, as we saw in section 3 and in sections 0, as well as in 1, $u(z(\eta))$ exists for every $z(\eta) \in ({{\delta}^{'}},{\delta}^{(2)} )$. And from this, as we saw in the
section 1, we can choose ${{\delta}^{'}}$, such that $0<{{\delta}^{'}} <{\delta}^{(2)}$. In particular, since the function $z(\eta)={{\delta}^{'}}\cosh(\eta)$ is strictly increasing in $(0, +\infty)$, we can finally obtain the interval of local existence of the solution $u({{\delta}^{'}}\cosh(\eta))$, i.e. for each $\eta \in \Big(0, \cosh^{-1}\Big(\frac{{\delta}^{(2)}}{{{\delta}^{'}}}\Big)\Big)$.

\quad

 Therefore from (3.11), unless we restrict $\Big(0, \cosh^{-1}\Big(\frac{{\delta}^{(2)}}{{{\delta}^{'}}}\Big)\Big)$, we can rewrite the solution, denoting as ${\mathcal{Q}}_{\lambda,\infty,n}(\eta):=u({{\delta}^{'}}\cosh(\eta))$ in a following way

\begin{equation}
	{\mathcal{Q}}_{\lambda,\infty,n}(\eta)=\sum_{i=0}^{+\infty}\frac{1}{i!}\frac{\partial^{i} u(a)}{\partial ({{\delta}^{'}}\cosh(\eta))^{i}}( {{\delta}^{'}}\cosh(\eta)-a)^{i},
\end{equation}
 
 with $a$ an opportune real value in $({{\delta}^{'}},{\delta}^{(2)} )$.

\quad
\quad

Thus, to sum up, we get the following theorem.

\begin{te}
Given the nonlinear operator (0.1)

$$ \Delta_H u^{n}=\frac{1}{\sinh \eta}\frac{\partial}{\partial \eta}\sinh \eta \frac{\partial  u^{n}}{\partial \eta}  \quad n>1.$$

There exists a ${\delta}^{(2)}>0$ such that $\forall {\delta}^{'} \in (0,{{\delta}^{(2)}} )$ there exists a single eigenfunction of the operator (0.1)$$ {\mathcal{Q}}_{\lambda,\infty,n}(\eta)=\sum_{i=0}^{+\infty}\frac{1}{i! }\frac{\partial^{i} u(a)}{\partial ({\delta}^{‘}\cosh(\eta))^{i}}( {\delta}^{’}\cosh(\eta)-a)^{i},$$ for a real eigenvalue $\lambda \neq 0$, in $\Big(0, \cosh^{-1}\Big(\frac{{\delta}^{(2)}}{{\delta}^{'}}\Big)\Big)$. 
					Where the number ‘$a$’ has been described in (3.11) and the function $u({\delta}^{‘}\cosh(\eta))$ with the variable $z(\eta) ={\delta}^{’}\cosh(\eta)$ is an analytic function in $(-{\delta}^{(2)},{\delta}^{(2)})$ such that
											
											\quad
											\quad
											
				\begin{equation}
				\frac{\partial^{k+2}u(z)}{\partial z^{k+2}}=\mathcal{Q}_{k+2,n}\Big(u(z),\frac{\partial u(z)}{\partial z}, \lambda,{\delta}^{'}\Big) \hspace{0.1cm} k>0,
				\end{equation}
												
			where $\mathcal{Q}_{k+2,n}$ is an algebraic rational function defined in proposition  (2.1) and by a recursive rule in (2.10).
											\end{te}
\section{Eigenfunction of $ \Delta_H u^{n}$ like a solution of a nonlinear hyperbolic partial differential equations }

\quad
\quad

Now, we return to the PDE (0.2).

$$	\widehat{O}_t  u (\eta,t)= \Delta_H  u^{n}-u = \frac{1}{\sinh \eta}\frac{\partial}{\partial \eta}\sinh \eta \frac{\partial  u^{n}}{\partial \eta}-u.$$

 Using the method of invariant vector spaces and the Theorem (3.1), we finally can get another solution apart from solutions of the type $u(\eta,t)=f(t)\sqrt[n]{c_1\ln(\tanh \frac{\eta}{2})+c_2}$, where $\widehat{O}_t f(t)=-f(t)$, that is $u(\eta,t)=f(t) 	{\mathcal{Q}}_{\lambda,\infty,n}(\eta)$ (see (3.11) and Theorem 3.1) such that 
\begin{equation}
	\widehat{O}_t f(t)=\lambda f^n(t) -f(t).
\end{equation}

In the case $\widehat{O}_t = \frac{\partial^\beta u}{\partial t^\beta}= \frac{1}{\Gamma(1-\beta)}\int_0^t (t-\tau)^{-\beta}\frac{\partial u}{\partial \tau}d\tau$  with $\beta \in (0,1)$, i.e $\widehat{O}_t $ is a Caputo fractional derivative of order
$\beta$(see [4] and [2]), then (4.1) becomes

\begin{equation}
	\frac{\partial^\beta f}{\partial t^\beta}=\lambda f^n(t) -f(t).
\end{equation}

While if $	\widehat{O}_t=\frac{\partial}{\partial t}t \frac{\partial}{\partial t}$ i.e a Laguerre derivative, we get from (4.1) an ODE equation in t several from (4.2) in the following way

\begin{equation}
\frac{\partial}{\partial t}t \frac{\partial}{\partial t}f(t)	=\lambda f^n(t) -f(t).
\end{equation}

And finally, if 	$\widehat{O}_t=\frac{\partial}{\partial t}$, i.e. the well-known partial derivative the (0.2) becomes a usual nonlinear PDE

\begin{equation}
 \frac{\partial u (\eta,t)}{\partial t}= \Delta_H  u^{n}-u = \frac{1}{\sinh \eta}\frac{\partial}{\partial \eta}\sinh \eta \frac{\partial  u^{n}}{\partial \eta}-u,
\end{equation}

and (4.1) becomes the following ODE in $t$

\begin{equation}
 \frac{\partial f(t)}{\partial t}	=\lambda f^n(t) -f(t).
\end{equation}

\quad
\quad

The last nonlinear PDE, which we'll illustrate in this section, involves
periodic solutions with respect to the time variable, i.e. $t$. It is given by the following:

		\begin{equation}
		\frac{\partial u}{\partial t}-\frac{i\omega}{\alpha}u = \frac{1}{\sinh \eta}\frac{\partial}{\partial \eta}\sinh \eta \frac{\partial  u^{n}}{\partial \eta}
	\end{equation}

(for this, see [2] and [7]), where, by separating the variables and applying the Theorem (3.1) we get, again a solution of the type $u(\eta,t)=f(t) 	{\mathcal{Q}}_{\lambda,\infty,n}(\eta)$ and $f(t)$ satisfies the following ODE equation respect to $t$

\begin{equation}
	\frac{\partial f}{\partial t}-\frac{i\omega}{\alpha}f=\lambda f^n(t).
\end{equation} 

\quad

In the both equations (4.4) and (4.6) we can explicitly determine $f(t)$. In fact, (4.5) and (4.7) belong to the following classes of ODE 
\begin{equation}
\frac{\partial f}{\partial t}=A_{1}f^n+A_{2}f, \hspace{0.1cm}with\hspace{0.1cm} n>1\hspace{0.1cm} and\hspace{0.1cm}A_{1} \neq 0,A_{2} \neq0 \in \mathbb{C}. 
\end{equation}

The equation (4.8) has a solution which depends on $A_{1}$, $A_{2}$ in the following way

\begin{prop}
	The ODE $$ \frac{\partial f}{\partial t}=A_{1}f^n+A_{2}f, \hspace{0.1cm}with\hspace{0.1cm} n>1\hspace{0.1cm} and\hspace{0.1cm}A_{1} \neq 0,A_{2} \neq0 \in \mathbb{C}.   $$ has got the solution of this type
	
	\begin{equation}
		f(t)=\sqrt[n-1]{\frac{-A_{2}}{A_{1}-ce^{-(n-1)A_{2}t}}},
	\end{equation}
where $c$ is a real constant, determined with an appropriate choice of initial conditions.

\end{prop}

\begin{proof}
If we calculate $\frac{\partial }{\partial t}	\Big(\sqrt[n-1]{\frac{-A_{2}}{A_{1}-ce^{-(n-1)A_{2}t}}}\Big)$, we get

$$ \frac{1}{n-1}\sqrt[n-1]{\Bigg(\frac{A_{1}-ce^{-(n-1)A_{2}t}}{-A_{2}}\Bigg)^{n-2} }\cdot  \frac{(-A_{2})^{2}(n-1)ce^{-(n-1)A_{2}t}}{(A_{1}-ce^{-(n-1)A_{2}t})^{2}} = $$
$$= \frac{-A_{2}ce^{-(n-1)A_{2}t}}{A_{1}-ce^{-(n-1)A_{2}t}}\cdot \sqrt[n-1]{\frac{-A_{2}}{A_{1}-ce^{-(n-1)A_{2}t}}}  . $$

Substituting the function (4.9) into the second member of the (4.8), we can derive

$$ A_{1}\Bigg(\sqrt[n-1]{\frac{-A_{2}}{A_{1}-ce^{-(n-1)A_{2}t}}}\Bigg)^{n}+A_{2}\sqrt[n-1]{\frac{-A_{2}}{A_{1}-ce^{-(n-1)A_{2}t}}}=$$

$$\hspace{1.5cm}= \frac{-A_{1}A_{2}}{A_{1}-ce^{-(n-1)A_{2}t}}\cdot \sqrt[n-1]{\frac{-A_{2}}{A_{1}-ce^{-(n-1)A_{2}t}}}+A_{2}\sqrt[n-1]{\frac{-A_{2}}{A_{1}-ce^{-(n-1)A_{2}t}}}$$

$$\hspace{-1.5cm}=A_{2} \sqrt[n-1]{\frac{-A_{2}}{A_{1}-ce^{-(n-1)A_{2}t}}} \Bigg( 1-\frac{A_{1}}{A_{1}-ce^{-(n-1)A_{2}t}}  \Bigg)$$

$$\hspace{-2.6cm}=\frac{-A_{2}ce^{-(n-1)A_{2}t}}{A_{1}-ce^{-(n-1)A_{2}t}}\cdot \sqrt[n-1]{\frac{-A_{2}}{A_{1}-ce^{-(n-1)A_{2}t}}}  , $$

is the same as the first member of (4.8), so (4.9) is a solution of (4.8).

\end{proof}

And now we can obtain a solutions to equations (4.4) and (4.6) by eigenfunctions of the nonlinear hyperbolic operator (0.1) in the following corollary.

\begin{coro}

Given a nonlinear hyperbolic PDEs (4.4) and (4.6)

$$   \frac{\partial u (\eta,t)}{\partial t}= \Delta_H  u^{n}-u = \frac{1}{\sinh \eta}\frac{\partial}{\partial \eta}\sinh \eta \frac{\partial  u^{n}}{\partial \eta}-u,     $$

$$	\frac{\partial u}{\partial t}-\frac{i\omega}{\alpha}u = \frac{1}{\sinh \eta}\frac{\partial}{\partial \eta}\sinh \eta \frac{\partial  u^{n}}{\partial \eta} .$$

For any real number $\lambda \neq 0$, given suitable initial conditions for (4.4) and (4.6), there exist locally the following solutions

\begin{equation}
u(\eta,t)=\sqrt[n-1]{\frac{1}{\lambda-ce^{(n-1)t}}}\cdot \sum_{i=0}^{+\infty}\frac{1}{i! }\frac{\partial^{i} u_{\lambda}(a)}{\partial ({\delta}^{'}\cosh(\eta))^{i}}( {\delta}^{'}\cosh(\eta)-a)^{i},    
\end{equation}

\begin{equation}
	u(\eta,t)=\sqrt[n-1]{\frac{-\frac{i\omega}{\alpha}}{\lambda-ce^{-(n-1)\frac{i\omega}{\alpha}t}}}\cdot \sum_{i=0}^{+\infty}\frac{1}{i! }\frac{\partial^{i} u_{\lambda}(a)}{\partial ({\delta}^{'}\cosh(\eta))^{i}}( {\delta}^{'}\cosh(\eta)-a)^{i},
\end{equation}
 where $c \neq 0$ is a suitable real number for the temporal part of $u(\eta,t)$ and ${\delta}^{'}>0$, and $a>0$ are appropriate real numbers such that the function $u_{\lambda}({\delta}^{'}\cosh(\eta))$ exists and is analytic near the initial conditions of (4.4) and (4.6) and enjoys the property (3.14).

\end{coro}
\begin{proof}
	
It's enough to apply Theorem 3.1 and Proposition 4.1, in particular for (4.10) $A_{1}=\lambda$ and $A_{2}=-1$, for (4.11) $A_{1}=\lambda$ and $A_{2}=\frac{i\omega}{\alpha}$.

\end{proof}

\end{document}